\documentclass[11pt, a4paper]{article}
\usepackage[a4paper, total={6.5in, 8.5in}]{geometry}
\usepackage[utf8]{inputenc} 
\usepackage[T1]{fontenc}
\usepackage[dvipsnames]{xcolor}
\usepackage{amsthm, amsmath, amssymb, amsfonts, graphicx, geometry, lipsum, amsxtra, array, enumitem, mathrsfs, verbatim, epigraph, tabularx,mathtools, esint, tcolorbox, authblk}
\usepackage[symbol]{footmisc}

\numberwithin{equation}{section}

\usepackage[dvipsnames]{xcolor}
\usepackage{etoolbox}
\usepackage{multirow}
\usepackage[hypertexnames=false]{hyperref}
\hypersetup{
    colorlinks=true,
    linkcolor=blue,
    citecolor=blue,
    urlcolor=red,
    pdftitle={Vanishing},
    }

\definecolor{titlepagecolor}{cmyk}{1,.60,0,.40}
\patchcmd{\subsection}{\normalfont}{\normalfont\color{black}}{}{}
\DeclareFixedFont{\titlefont}{T1}{ppl}{b}{it}{0.5in}
\def\th@plain{%
  \thm@notefont{}
  \itshape 
}
\def\th@definition{%
  \thm@notefont{}
  \normalfont 
}
\makeatother

\usepackage{srcltx} 
\usepackage{amscd}
\usepackage{pb-diagram}
\usepackage[all]{xy}
\xyoption{matrix}
\xyoption{arrow}
\usepackage[normalem]{ulem} 
\usepackage{pdfsync}
 
\usepackage{thmtools}

\DeclareUnicodeCharacter{2212}{-,+}
\usepackage[none]{hyphenat}

\theoremstyle{plain} 
\newtheorem{defn}{Definition}[section]
\newtheorem{theorem}[defn]{Theorem}

\newtheorem{lemma}[theorem]{Lemma}
 \newtheorem{prop}[defn]{Proposition}

\newtheorem{Theorem}{Theorem}

\title{Vanishing elements of prime power order and their class size property}
\author{Sonakshee Arora\textsuperscript{\textdagger} and Rahul Dattatraya Kitture\textsuperscript{$\ddagger$}}

\date{}
 
\begin{document}
\maketitle 

\begin{abstract} Study of the structure of groups by variation in the arithmetic conditions on conjugacy classes and character degrees has produced several interesting results and open problems. In continuation of such work, Dolfi and Lucido in \cite{MR1826493} introduced a property for groups. For primes $p,q$, a group $G$ is said to have property $P(p,q)$ if every $p'$-element in $G$ has $q'$-class size. They obtained several results on the structure of $G$ and of some subgroups when $G$ satisfies the property $P(p,q)$. Motivated by this work, we introduce a vanishing analogue of the above property: for primes $p \neq q$, a finite group $G$ is said to have the property $P_v(p,q)$ if every vanishing $p'$-element of prime power order in $G$ has conjugacy class size not divisible by $q$. We show that no finite simple group satisfies the property $P_v(p,q)$ for primes $p\neq q$ dividing $|G|$. We use this result to show that if a finite group $G$ satisfies the property $P_v(p,q)$ with $p \neq q$ and $p > 2$, then $O^{q'}(G)$ (subgroup generated by all Sylow $q$-subgroups of $G$) is solvable. This generalises a result of Dolfi and Lucido under weaker conditions.

\vskip3mm
\noindent{\bf Keywords.} Vanishing elements, irreducible characters, simple groups, conjugacy classes
		
\noindent{\bf Mathematics Subject Classification (2020)}: 20C15, 20E45, 20E32
\end{abstract}
\section{Introduction}
\noindent 
Different arithmetic conditions on the conjugacy classes of a group have led to interesting results about its structure. A glance at the survey article \cite{MR2875589} shows development on the subject. 

On the other hand, the character table of a group has been extensively used to determine the structure of the group. Many results about the structure of the group from conditions on conjugacy classes have analogues in terms of character degrees. A famous theorem in this direction is given by Camina in $1972$ (\cite{MR294481}, Corollary 1) in which he proved: \textit{if a prime $p$ does not divide the size of any conjugacy class of $G$, then $G$ has a central Sylow $p$-subgroup.} A dual version of this result, proved by Ito in $1951$ (\cite{MR41852}, Proposition 5) for solvable groups and Michler in $1986$ (\cite{MR842482}, Theorem 5.4,) for general groups, states: \textit{If a prime $p$ does not divide the degree of any irreducible character of $G$, then $G$ has an abelian normal Sylow $p$-subgroup.} Here, the characters are $\mathbb{C}$-irreducible characters. In $1986$ ( \cite{MR842482}, Theorem 5.5), Michler proved an analogue of his stated theorem for $p$-Brauer characters: \textit{A finite group $G$  has a normal Sylow $p$-subgroup ($p>2$) if and only if $p$ does not divide the degree of any irreducible $p$-Brauer character of $G$.} 

Manz and Wolf in 1988 (\cite{MR937602}, Theorem 2.9) obtained the following variation of Michler's last result stated above, considering two primes: \textit{Let $p\neq q$ be primes. If $q$ does not divide the degree of all $p$-Brauer characters of $G$, then $O^{q'}(G)$ is solvable.} (Here, $O^{q'}(G)$ is the subgroup generated by all the Sylow $q$-subgroups of $G$.) Motivated by this result and duality in character table between conjugacy classes and character degrees, Dolfi and Lucido in 2001 (\cite{MR1826493}, Theorem 3) obtained the following analogue of the theorem of Manz and Wolf: \textit{If $q$ does not divide the size of conjugacy class of every $p'$-element of $G$, then $O^{q'}(G)$ is solvable}. 

We prove the theorem of Dolfi and Lucido under a quite weaker assumption, in which we restrict to conjugacy classes of \textbf{vanishing} $p'$-elements of \textbf{prime power order}. For simplicity, we define a property, analogous to the one defined by Dolfi and Lucido in \cite{MR1826493}, that we recall here: \textit{a group $G$ is said to satisfy the property $P(p,q)$ if $q$ does not divide the class size of any $p'$-element in $G$.}

\vskip5mm\noindent 
\textbf{Definition}:\textit{ A group $G$ is said to satisfy the property $P_{v}(p,q)$ if $q$ does not divide the class size of any vanishing $p'$-element of prime power order in $G$.}

\vskip5mm
As an example, let $G=C_7\rtimes C_6$ with faithful action of $C_6$ on $C_7$; it is a Frobenius group of order $42$. Then $G$ satisfies the property $P_v(2,3)$, but does not satisfy the property $P(2,3)$. 
\vskip5mm
Our main results are the following. 

\begin{Theorem}\label{thma}
A non-abelian finite simple group $G$ does not satisfy the property $P_v(p,q)$, for any prime divisors $p\neq q$ of $|G|$.
\end{Theorem}

\vskip3mm
\noindent
By Theorem \ref{thma}, we obtain an analogue of the theorem of Dolfi and Lucido mentioned above (\cite{MR1826493}, Theorem 3). 
\vskip-5mm\noindent
\begin{Theorem}\label{thmb}
If $G$ is a finite group, satisfying the property $P_v(p,q)$ with primes $p\neq q$, and $p>2$, then $O^{q'}(G)$ is solvable.
\end{Theorem}

In defining the property $P_v(p,q)$, we allow $p,q$ to be any primes, not necessarily dividing $|G|$. This will help us to study the structure of $O^{q'}(G)$ by induction, specifically in the proof of Theorem \ref{thmb}. The proofs of both theorems depend on the classification of finite simple groups.

\vskip2mm
\noindent\textbf{Notations:}
\noindent All the groups considered here are finite. For $x,g\in G$, we write $x^g=gxg^{-1}$. The conjugacy class (resp. centralizer) of $x$ in $G$ is denoted by $x^G$ (resp. $C_G(x)$). If $G$ acts on a set $X$, and $Y\subseteq X$, then $\mathrm{Stab}_{G}(Y)$ denotes the stabilizer in $G$ of $Y$. For notations of simple groups, we refer to ATLAS \cite{MR827219} and Wilson (\cite{MR2562037}, p.3). Other notations are standard (see \cite{MR1280461}, \cite{MR569209}).

\section{Preliminaries}
We begin this section by stating some known results that we will use throughout the paper.

\begin{lemma}(\cite{MR2262862}, Lemma 5)\label{lem2}
Let $G$ be a finite group and $M=S_1\times\cdots\times S_k $ a minimal normal subgroup of $G$, where every $S_i$ is isomorphic to a non-abelian simple group $S$. If $\theta\in \mathrm{Irr}(S)$ extends to $\mathrm{Aut}(S)$, then $\theta\times\cdots\times \theta\in \mathrm{Irr}(M)$ extends to $G$.
\end{lemma}

\begin{lemma}(\cite{MR3918622}, Lemma 2.2)\label{Brough_van} Let $G$ be a group, and $N$ a normal subgroup of $G$. If $N$ has an irreducible character of $q$-defect zero, then every element of $N$ of order divisible by $q$ is a vanishing element in $G$.
\end{lemma}

A group $G$ is said to be $p$-nilpotent if it has a normal $p$-complement.
\begin{theorem}(\cite{MR3918622}, Theorem 3)\label{Sotomayor_van} Let $G = AB$ be a core-factorization.
\begin{enumerate}
\item Assume that $p$ does not divide $|x^G|$ for every $p'$-element of prime power order
$x\in A\cup B$ vanishing in $G$. Then, $G$ is $p$-nilpotent. 
\item If $p$ does not divide $|x^G|$ for every prime power order element $x\in A\cup B$ vanishing in
$G$, then $G$ is $p$-nilpotent with abelian Sylow $p$-subgroups.
\end{enumerate}
\end{theorem}

In other words if $G$ satisfies the property $P_v(p,p)$, then $G$ is $p$-nilpotent.

\begin{theorem}(\cite{MR2469367}, Theorem A)\label{Thm_normal} Let $G$ be a finite group and $p$ a prime number. If all the $p$-elements of G are non-vanishing, then $G$ has a normal Sylow $p$-subgroup.
\end{theorem}
\begin{prop}(\cite{MR3918622}, Proposition 2) Let $N$ be a non-abelian minimal normal subgroup of a finite group $G$, and $p$ a prime divisor of $|N|$. Then, there is a $p$-element in $N$ which is vanishing in $G$.
\end{prop}

\begin{theorem}(\cite{MR1402885}, Theorem 3)\label{zsigmondy}
Let $a$ and $n$ be integers greater than $1$. There exists a prime divisor $q$ of $a^n-1$ such that $q$ does not divide $a^j-1$ for all $j$ with $0<j<n$, except precisely in the following cases: $n=6$, $a=2$ or 
$n=2$, $a=2^s-1$, where $s\geq 2$. 
Such a prime $q$, when it exists, is called a \textit{primitive prime divisor of $a^n-1$}.
\end{theorem}

\section{Proof of Theorem \ref{thma}}
The aim of this section is to prove that a non-abelian finite simple group does not satisfy the property $P_v(p,q)$ for any prime divisors $p\neq q$ of its order. We first establish this result for some specific families of simple groups, namely, for $A_n$ ($n\ge 6$), a family of orthogonal groups of plus type, and some sporadic simple groups. Then, we prove the same for an arbitrary simple group according to the number of connected components of its prime graph.

\begin{lemma}\label{lem_alt}
Let $S\cong A_n (n\geq 9), M_{12}, J_2, HS, Suz, Ru, Co_1, Co_3, or BM$  and $p$ a prime divisor of $|S|$. Then there exists a $p'$-element $x\in S$ of prime power order and $\chi\in\mathrm{Irr}(S)$ such that $|x^S|$ is divisible by every prime divisor of $|S|$, $\chi$  extends to $\mathrm{Aut}(S)$ and $\chi(x)=0$. In particular, $S$ does not satisfy $P_v(p,q)$. 
\end{lemma}

\begin{proof} 
For suitable prime divisors $p_1\neq p_2$ of $|S|$, we find $p_i$-element $x_i\in S$ ($i=1,2$) and $\chi_i\in\mbox{Irr}(S_i)$ such that $|x_i^S|$ is divisible by every prime divisor of $|S|$, $\chi_i(x_i)=0$, and $\chi_i$ extends to $\mbox{Aut}(S_i)$. Then, for a given prime $p$, one of the $x_i$ will be a $p'$-element satisfying the assertion.

\vskip2mm 

Let $S\cong A_n$ with $n\ge 9$. Suppose $p=2$. Let $r$ be the largest prime such that $r^2\le n$. Then $r$ is odd. By (\cite{Arora03092025}, Theorem 3.2), there exists a $r$-element $x\in S$ (which is a $2'$-element) and $\chi\in\mbox{Irr}(S)$ satisfying the assertion in the Lemma.
Suppose $p>2$. We show that there is $x\in A_n$ of order $4$ or $8$ (a $p'$-element) satisfying the assertion in the Lemma.  Consider the class functions $\pi$ and $\delta$ on $A_n$ given by 
\begin{align*}
\pi(g) & =\mbox{number of fixed points of } g,\\
\delta(g) & = \mbox{number of transpositions in disjoint-cycle decomposition of } g \hskip5mm (g\in A_n).
\end{align*}
Since $n\ge 9$, by (\cite{MR1645304}, Theorem 11.9),  the class functions $\sigma$ and $\tau$ defined below are irreducible characters of $A_n$, and they extend to $\mbox{Aut}(A_n)$: 
$$
\sigma:=\frac{\pi(\pi-3)}{2}+\delta
\quad \text{and} \quad
\tau:=\frac{(\pi-1)(\pi-2)}{2}-\delta,
$$
Now $n=4k+t$, where $t\in\{0,1,2,3\}$. Since $n\ge 9$, we get $k\ge 2$. We take $x\in A_n$ as given below: 

(i) If $k$ is even, choose $x\in A_n$ which is product of $k$ disjoint $4$-cycles. 

(ii) If $k$ is odd and $n\in \{4k, 4k+1\}$, choose $x\in A_n$ which is product of $k-2$ disjoint  $4$ cycles and an $8$-cycle.  

(iii) $k$ is odd and $n\in \{4k+2, 4k+3\}$, choose $x\in A_n$ which is product of $k$ disjoint $4$ cycles and a transposition. 

The choice of $x$ ensures that either $\sigma(x)=0$ or $\tau(x)=0$ in each case. Also, in cases (i), (ii) and (iii) respectively, 
$$
|x^{A_n}|=\frac{n!}{4^k k!}
\quad \text{or} \quad
\frac{n!}{4^{k-2}8(k-2)!}
\quad \text{or} \quad
\frac{n!}{4^k k!2}.
$$
\textbf{Claim:} Every prime divisor of $n!$ divides $|x^{A_n}|$.

Notice that (for fixed $k\ge 2$), the GCD of above three integers is the last term. Hence, it suffices to show: every prime divisor of $n!$ divides $\frac{n!}{4^kk!2}$.

Since $k!4^k$ divides $(4k)!$, we write $(4k)!=N4^kk!$, where $N$ is the product of all $c\in \{1,\ldots, 4k\}$ such that $c\not\equiv 0\pmod{4}$. 

If $l$ is prime dividing $(4k)!$, then $l\not\equiv 0\pmod{4}$, hence $l$ divides $N$. Therefore, $l$ divides $\frac{(4k)!}{4^k k!}$, which divides $\frac{n!}{4^k k!}$. Since $2\cdot 2k$ divides $N$, it follows that $l$ divides $n!/(4^k k! 2)$. 

\vskip5mm
For $S\not\cong A_n$, with the notations in \cite{MR827219}, the table below provides, for each prime divisor $p$ of $|S|$, a suitable $p'$-element $x\in S$ and $\chi\in \mathrm{Irr}(S)$, satisfying the assertion.
\begin{center}
\begin{tabular}{|c|c|c|c|c|c|c|c|c|c| } 
\hline
Prime & $S$ &  $M_{12}$ & $J_2$ & HS & Suz & Ru & $Co_1$ & $Co_3$ & BM \\\hline
$p=2$ & $x$  & $3B$  &  $3A$ & $5C$ & $9A$ & $5B$ & $3D$ & $3B$ & $5B$ \\
 & $\chi$  & $\chi_7$   & $\chi_6$ & $\chi_7$ & $\chi_4$ & $\chi_9$ & $\chi_2$ & $\chi_{21}$ & $\chi_5$ \\\hline 
$p>2$ &  $x$  & $8A$ & $2B$ & $4A$ & $8A$ & $8A$ & $4F$ & $8C$ & $8N$ \\
& $\chi$  & $\chi_{13}$ & $\chi_6$ & $\chi_{18}$ & $\chi_3$ & $\chi_2$ & $\chi_{2}$ & $\chi_{9}$ & $\chi_{13}$ \\\hline
\end{tabular}
\end{center}

\end{proof}

\begin{lemma}\label{cor2}
  Let $S\cong A_6,\, A_7,\, A_8,\, M_{22}, \text{ or }M_{24}$ and $p\neq q$ be prime divisors of $|S|$. Then there exists a $p'$-element $x\in S$ of prime power order and $\chi\in \mathrm{Irr}(S) $ such that $|x^S|$ is divisible by $q$, $\chi$ extends to $\mathrm{Aut}(S)$ and $\chi(x)=0$. In particular, $S$ does not satisfy the property $P_v(p,q)$. 
\end{lemma}

\begin{proof}  Following the notations in \cite{MR827219}, the required $x$ and $\chi$ are mentioned in the table below. 
\begin{center}
\begin{tabular}{|c|c|c|c|c|c|c|}\hline 
Primes & $S$  &  $A_6$ & $A_7$ & $A_8$ & $M_{22}$ & $M_{24}$ \\\hline 
$p=2$, $q\neq 5$  &$x$  & $5A$ & $5A$ & $5A$ & $5A$ & $5A$ \\
 & $\chi$ &  $\chi_7$ &$\chi_9$ & $\chi_{14}$ & $\chi_8$ & $\chi_{26}$ \\\hline 
 $p=2$, $q\neq 7$ & $x$ & $3A$ &$7A$ & $7A$ & $7A$ & $7A$ \\
 & $\chi$ & $\chi_6$ & $\chi_9$ & $\chi_{14}$ & $\chi_8$ & $\chi_{26}$ \\\hline
 $p>2$ & $x$ &$4A$ & $4A$ & $4A$ & $8A$ & $8A$\\
 & $\chi$ & $\chi_7$ &$\chi_2$ & $\chi_4$ & $\chi_7$ & $\chi_7$ \\\hline
 \end{tabular}
\end{center}
\end{proof}


\begin{lemma}\label{lemma_orthogonal} 
Let $S=P\Omega_n^+(p^r)$ ($p$ a prime, $r\ge 1$) with $n\ge 8$ and $n\equiv 0\pmod{4}$, and let  $q\neq p$ a prime divisor of $|S|$. Then $S$ does not satisfy $P_v(p,q)$. 
\end{lemma}
\begin{proof} By (\cite{MR1321575}, Corollary 2), $S$ has an irreducible character of defect zero for every prime divisor of $|S|$. Hence, all the non-identity elements of $S$ are vanishing. If $S=P\Omega_8^+(2)$, then from (\cite{MR827219}, p.86), we see that the elements of order $4$ and $9$ have class size divisible by every prime divisor of $|S|$. Hence $S$ does not satisfy $P_v(p,q)$. Assume that $S\not\cong P\Omega_8^+(2)$.

For simplicity, write $t:=p^r.$
By (\cite {MR2507573}, Theorem 4.1), there exist maximal tori $T_1$ and $T_2$ in $S$ of following coprime orders:
$$
|T_1|=\frac{(t^{\frac{n}{2}-1} -1)(t-1)}{(2,t-1)^2}, \hskip10mm 
|T_2|=\frac{(t^{\frac{n}{2}-1} +1)(t+1)}{(2,t-1)^2}
$$
Let $l$ be a primitive prime divisor of $t^{(\frac{n}{2}-1)}-1$ (this exists by Theorem \ref{zsigmondy}), since $\frac{n}{2}-1\notin \{2,6\}$ for $n$ given in the hypothesis). Choose $\lambda\in\overline{\mathbb{F}}_t^*$ of order $l$ and choose $t_1\in T_1$ which  diagonalizes over $\overline{\mathbb{F}}_t$ as follows: 
\begin{align*}
t_1 &\sim \mbox{diag}(\lambda, \lambda^t, \ldots, \lambda^{t^{\frac{n}{2}-2}}, 
\lambda^{-1}, \lambda^{-t}, \ldots, \lambda^{-t^{\frac{n}{2}-2}}, 1,1).
\end{align*}
Let $k$ be a primitive prime divisor of $t^{(n-2)}-1$; this divides $t^{(\frac{n}{2}-1)}+1$ (such $k$ exists since 
$(t,n-2)\neq (2,6)$ by the assumptions made in the beginning, and $n-2\neq 2$ by the hypothesis; see Theorem \ref{zsigmondy}).  Choose $\mu\in\overline{\mathbb{F}}_t^*$ of order $k$ and choose $t_2\in T_2$ which  diagonalizes over $\overline{\mathbb{F}}_t$ as follows: 
\begin{align*}
t_2 &\sim \mbox{diag}(\mu, \mu^t, \ldots, \mu^{t^{n-3}}, 1,1).
\end{align*}
Note that $t_i$ are $p'$-elements (since $T_i$ contains semisimple elements in $S$).

The choices of $\lambda$ and $\mu$ of orders $l$ and $k$ respectively ensure that the first $n-2$ eigenvalues in the above diagonalisation of each $t_i$ are distinct. Following the proof of Lemma 4.6 in \cite{MR2507573}, we get that 
$C_S(t_i)=T_i$ for $i=1,2$. Since $(|T_1|, |T_2|)=1$, $q$ is coprime with one of $|T_i|$, hence $q$ divide $|t_i^S|$ for some $i$. Therefore, $S$ does not satisfy the property $P_v(p,q)$.
\end{proof}

For a finite group $G$, its prime graph is a graph whose vertices are the prime divisors of $|G|$, and two vertices $p_i\neq p_j$ are connected if there is an element in $G$ of order $p_ip_j$. The prime graphs of finite simple groups are well known, and we follow \cite{MR1922988}, \cite{MR617092}, and \cite{MR1015040} for our next results. The connected components of the prime graph of $G$ are denoted by $\pi_i(G)$ ($1\le i\le k$), and $\pi_1(G)$ denotes the connected component containing $2$ (provided $2$ divides $|G|$). 

\subsection*{Simple groups with at least $3$ components}
\begin{lemma}\label{lemma2.3}
Let $S$ be a non-abelian finite simple group with at least $3$ connected components in its prime graph, and $p\neq q$ be primes dividing $|S|$. Then $S$ does not satisfy $P_v(p,q)$. 
\end{lemma}
\begin{proof}
Let $\pi_i(S)$ ($i=1,\ldots,k$) denote the connected components of the prime graph of $S$. For $p\in \pi_i(S)$ and $q\in \pi_j(S)$, choose $t\in \pi_k(S)$ where $k\notin \{i,j\}$. If all the $t$-elements in $S$ are non-vanishing, then by Theorem \ref{Thm_normal}, $S$ has a normal Sylow $t$-subgroup, a contradiction. Hence, $S$ contains a vanishing $t$-element say $x$, and since $q,t$ are in different connected components, $q\mid |x^S|$.      
\end{proof}

\subsection*{Simple groups with 2 components}
By (\cite{MR1922988}, p.95) the non-abelian simple groups whose prime graph has 2 connected components are 

(i) $A_n$ for \textit{some} $n>6$; (ii) some simple groups of Lie type; (iii) the sporadic groups $M_{12}$,$J_2$, $Co_1$, $Co_3$, $McL$, $He$, $HN$, $Fi_{22}$, $Ru$. 

\vskip5mm
It should be noted that not all the alternating groups or simple groups of Lie type have exactly two connected components in their prime graph.

\begin{lemma}\label{lemma2.4}
Let $S$ be a non-abelian finite simple group whose prime graph has two connected components, and $p\neq q$ be primes dividing $|S|$. Then $S$ does not satisfy $P_v(p,q)$. 
\end{lemma}
\begin{proof}
As mentioned before the Lemma, $S$ is among the families (i), (ii), and (iii) above. By Lemma \ref{lem_alt} and Lemma \ref{cor2}, we many assume that $S\not\cong A_n$, and $S\not\cong M_{12}, J_2, Co_1, Co_3, Ru$. Thus, 
\begin{center}
$S$ is a simple group of Lie type or $S\cong McL, He, HN$ or $Fi_{22}$. 
\end{center}
Each $S$ above has an irreducible character of defect zero for every prime divisor of $|S|$. Hence, by (\cite{MR1280461}, Theorem 8.17), all the non-identity elements of $S$ are vanishing. Let $\pi_1(S)$ and $\pi_2(S)$ be the connected components of the prime graph of $S$, with $2\in \pi_1(S)$.

Suppose $S\cong McL, He, HN$ or $Fi_{22}$. By (\cite{MR1922988}, Table 1a), $\pi_2(S)=\{d\}$ for some odd prime $d$ dividing $|S|$. Note that $|\pi_1(S)|\ge 2$. Since $p\neq q$, at most one of $p,q$ can lie in $\pi_2(S)$.

(i) $q\in\pi_2(S)$ (i.e. $q=d$): then for any $s$-element for $s\in\pi_1(S)\setminus \{p\}$, we have $q\mid |s^S|$.  

(ii) $q\in \pi_1(S)$ but $p\notin \pi_2(S)$ (i.e. $p,q\in \pi_1(S)$): then for any $d$-element $t\in S$, we have $q\mid|t^S|$.   

(iii) $q\in \pi_1(S)$ and $p\in\pi_2(S)$ (i.e. $p=d$): in this case, with the notations as in \cite{MR827219}, we choose a $2$-element $x$ in $S$ from classes  $4A, 4B, 4C,$ and $2C$ in the groups $McL, He, HN,$ and $Fi_{22}$ respectively. From \cite{MR827219}, we get that $|x^S|$ is divisible by every prime divisor of $|S|$. Hence $S$ does not satisfy $P_v(p,q)$
\vskip5mm
Suppose $S$ is a simple group of Lie type and $U$ be a unipotent subgroup of $S$. Note that $U$ is a Sylow subgroup of $S$ and $C_S(U)\subseteq U$.

\textbf{Case I.} $p$ divides $|U|$.

Here, $p$ is the characteristic of the field of definition of $S$. Hence, by (\cite{MR1922988}, Table 1), $p\in\pi_1(S)$.  

{\bf (a)} If $q\in\pi_1(S)$, then for any $s\in \pi_2(S)$, an $s$-element has class size divisible by $q$. 

{\bf (b)} If $q\in \pi_2(S)$, then for any $s\in\pi_1(S)\setminus\{p\}$, any $s$-element has class size divisible by $q$.  

In both cases, we get that $S$ does not satisfy the property $P_v(p,q)$.

\vskip3mm
\textbf{Case II.} $p$ is coprime with $|U|$.

Assume that $q\mid |U|$. Then $U$ is a Sylow $q$-subgroup of $S$. For any prime $s\notin \{p,q\}$ dividing $|S|$,  let $x\in S$ be an $s$-element; it is semisimple. If $q\nmid |x^S|$, then $C_S(x)\supseteq gUg^{-1}$ (for some $g\in S$), i.e. $C_S(g^{-1}xg)\supseteq U$, i.e. $g^{-1}xg \in C_S(U)\subseteq U$, a contradiction. Hence, $q\mid |x^S|$ and $S$ does not satisfy the property $P_v(p,q)$. 

Assume $q$ is coprime with $|U|$ and $u\in U$ be a regular unipotent element (it is a $p'$-element). By (\cite{MR794307}, Proposition 5.1.5), the only semisimple element in $C_S(u)$ is $1$. If $q\nmid |u^S|$, then $C_S(u)$ contains a Sylow $q$-subgroup of $S$, and since $q$-elements are semisimple, we arrived at a contradiction. Hence $q||u^S|$ and $S$ does not satisfy the property $P_v(p,q)$.  
\end{proof}

\subsection*{Simple groups with one connected component}
By (\cite{MR617092}, Table Ia) and (\cite{MR1015040}, Table 1), the only non-abelian simple groups whose prime graph is connected are 

(i) $A_n$ for some $n\geq 5$; (ii) some classical simple groups; (iii) $E_7(t)$ with $t>3$. 

\vskip3mm 
Note that not all $A_n$ or classical simple groups have connected prime graphs (for example, $A_p$ where $p$ is a prime, and $\mathrm{PSL}_2(q)$ where $q-1$ or $q+1$ is a prime). 
 
\begin{lemma}\label{lemma2.6}
Let $S$ be a non-abelian finite simple group whose prime graph is connected, and $p\neq q$ be prime divisors of $|S|$. Then $S$ does not satisfy the property $P_v(p,q)$. 
\end{lemma}
\begin{proof} The simple groups with connected prime graphs are among the families (i)-(iii) mentioned before the lemma.  

If $S\cong A_5$, then for $p=2$ take $x\in A_5$ from conjugacy classes of elements of order $3$ and $5$. Their conjugacy class size is $2^2\cdot 5$ and $2^2\cdot 3$ respectively. For $p=3$, take $x\in A_5$ from conjugacy class of elements of order $2$ and $5$. Their conjugacy class size is $3\cdot 5$ and $2^2\cdot 3$ respectively. 
For $p=5$, take $x\in A_5$ from conjugacy class of elements of order $2$ and $3$. Their conjugacy class size is $3\cdot 5$ and $2^2\cdot 5$ respectively. It follows that $A_5$ do not satisfy $P_v(p,q)$ for any $\{p,q\}$. By Lemma \ref{lem_alt} and Lemma \ref{cor2}, we may assume that $S\not\cong A_n$ ($n\geq 6$). 
If $S$ is a simple group of Lie type defined over a field of characteristic other than $p$, then the arguments in Case II of Lemma \ref{lemma2.4} (which do not rely on the prime graph of $S$) imply that $S$ does not satisfy the property $P_v(p,q)$. Thus, we assume that 
\begin{center}
$S$ is a classical simple group over a field of characteristic $p$ \hskip5mm or \hskip5mm $S\cong E_7(p^r)$ 
\end{center}
By Lemma \ref{lemma_orthogonal}, we may assume that $S\not \cong P\Omega_n^+(p^r)$ where $n\equiv 0\pmod{4}$ and $n\ge 8$. By (\cite{MR2507573}, Theorem 4.1), there exist two maximal cyclic tori $T_1$ and $T_2$ of coprime orders in $S$, whose orders are as below: 

\vskip4mm
\begin{center}
\renewcommand{\arraystretch}{1.4}
\begin{tabular}{|c|c|c|c|c|}
\hline
$S$ & Conditions & $d$ & $|T_1|$ & $|T_2|$\\
\hline
\multirow{2}{*}{$\mathrm{PSL}_n(p^r)$}
 & $n\ge 2$; & \multirow{2}{*}{$(n,p^r-1)$}
 & \multirow{2}{*}{$\dfrac{1}{d}\dfrac{p^{nr}-1}{p^r-1}$}
 & \multirow{2}{*}{$\dfrac{1}{d}\left(p^{(n-1)r}-1\right)$}\\
 & $p^r\ge 4$ if $n=2$ & & & \\
\hline
\multirow{2}{*}{$\mathrm{PSU}_n(p^r)$}
 & $n\ge 3$ odd;  & \multirow{2}{*}{$(n,p^r+1)$}
 & \multirow{2}{*}{$\dfrac{1}{d}\dfrac{p^{nr}+1}{p^r+1}$}
 & \multirow{2}{*}{$\dfrac{1}{d}\left(p^{(n-1)r}-1\right)$}\\
 & $p^r\ge 3$ if $n=3$ & & & \\
\hline
$\mathrm{PSU}_n(p^r)$ & $n\ge 4$ even & $(n,p^r+1)$
 & $\dfrac{1}{d}\dfrac{p^{nr}-1}{p^r+1}$ & $\dfrac{1}{d}\left(p^{(n-1)r}+1\right)$\\
\hline
\multirow{2}{*}{$\mathrm{PSp}_n(p^r)$}
 & $n\ge 4$ even; & \multirow{2}{*}{$(2,p^r-1)$}
 & \multirow{2}{*}{$\dfrac{1}{d}\left(p^{\frac{nr}{2}}+1\right)$}
 & \multirow{2}{*}{$\dfrac{1}{d}\left(p^{\frac{nr}{2}}-1\right)$}\\
 &$p^r\ge 3$ if $n=4$ & & & \\
\hline
$\mathrm{\Omega}_n(p^r)$ & $n\ge 7$ odd;  & \multirow{2}{*}{$2$}
 & \multirow{2}{*}{$\dfrac{1}{d}\left(p^{\frac{(n-1)r}{2}}+1\right)$} & \multirow{2}{*}{$\dfrac{1}{d}\left(p^{\frac{(n-1)r}{2}}-1\right)$}\\
 &$p^r$ odd&&&\\
\hline
$\mathrm{P\Omega}_n^{-}(p^r)$ & $n\ge 8$ even & $(4,p^{\frac{nr}{2}}+1)$
 & $\dfrac{1}{d}\left(p^{\frac{nr}{2}}+1\right)$
 & $\dfrac{1}{d}\left(p^{\frac{nr}{2}-1}+1\right)(p^r-1)$\\
\hline
\multirow{2}{*}{$\mathrm{P\Omega}_n^{+}(p^r)$}
 & $n\ge 10$, & \multirow{2}{*}{$(4,p^{\frac{nr}{2}}-1)$}
 & \multirow{2}{*}{$\dfrac{1}{d}\left(p^{\frac{nr}{2}}-1\right)$}
 & \multirow{2}{*}{$\dfrac{1}{d}\left(p^{\frac{nr}{2}-1}+1\right)(p^r+1)$}\\
 & $n\equiv 2\pmod 4$ & & & \\
\hline
$\mathrm{E}_7(p^r)$ & & $(2,p^r-1)$ & $\dfrac{1}{d}\left(p^{7r}-1\right)$ & $\dfrac{1}{d}\left(p^{7r}+1\right)$\\
\hline
\end{tabular}
\end{center}

\vskip5mm
Let $T_j=\langle x_j\rangle$ for $j\in\{1,2\}$. Since $(p,|T_j|)=1$, the elements of each of $T_j$ are $p'$-elements. 

Now $q\neq p$, i.e. $q$ is different from the characteristic of the field for $S$. Since $|T_1|$ and $|T_2|$ are coprime, $q$ is coprime with one of these orders, say $$(q,|T_i|)=1.$$ 
For each $S$ in the above table, the numerator in the expression of $|T_i|$ is of the form $p^c \pm 1$, where $c\ge 1$. In case of $\mathrm{P\Omega}_n^{+}(p^r)$ and $\mathrm{P\Omega}_n^{-}(p^r)$, we take the numerator of $|T_2|$ to be $p^{\frac{nr}{2}-1}+1$. 
\vskip2mm
\noindent\textbf{Case 1.} The numerator in the expression of $|T_1|$ and $|T_2|$ have a primitive prime divisor:

\noindent Choose $a_i\in T_i$ such that $o(a_i)$ is a primitive prime divisor of the numerator in the expression of $|T_i|$. Let $\widetilde{S}$ be the quasi-simple group corresponding to $S$ (so $\widetilde{S}/Z(\widetilde{S})=S$). Let $\widetilde{T}_i\le \widetilde{S}$ be such that $\widetilde{T}_i/Z(\widetilde{S})=T_i$.
The primitivity of $o(a_i)$ ensures that 
$$
(o(a_i), |Z(\widetilde{S})|)=1.
$$
Thus, we choose $\tilde{a_i}\in \widetilde{S}$ of order equal to $o(a_i)$, and  it is easy to see that
$$
C_S(a_i)=C_{\widetilde{S}}(\tilde{a}_i)/Z(\widetilde{S}).
$$
In particular, $|a_i^S|=|\tilde{a_i}^{\widetilde{S}}|$. By (\cite{MR2507573}, Lemma 4.6), $C_{\widetilde{S}}(\tilde{a}_i)=\widetilde{T}_i$. It follows that 
$$
|C_{S}(a_i)|=|T_i|
$$
which is coprime with $q$. Therefore, $q$ divides $|a_i^S|$, where $a_i$ is a $p'$-element of prime order.

\vskip5mm

\noindent\textbf{Case 2.} The numerator in the expression of $|T_1|$ or $|T_2|$ has no primitive prime divisor:

By Theorem \ref{zsigmondy}, $p^c+1$ has no primitive prime divisor if and only if
$$
(p=2,  \,\,\, 2c=6) \hskip1cm\mbox{ or } \hskip1cm (p=2^s-1 \mbox{ where }s\ge 2, \,\,\, 2c=2);
$$
and  $p^c-1$ has no primitive prime divisor if and only if
$$
(p=2,  \,\,\, c=6) \hskip1cm\mbox{ or } \hskip1cm (p=2^s-1 \mbox{ where }s\ge 2, \,\,\, c=2).
$$ 
For each $S$ in the table above, with the given conditions, Case $2$ occurs only for the following groups:

 \begin{center}
\begin{tabular}{|c|c|c|}\hline 
$S$ & $(n,r)$ & Condition on $p$\\\hline
$\mathrm{PSL}_n(p^r)$ & $(2,1), \,(2,2),\,(3,1)$ & $p=2^s-1$ with $s\geq 2$ \\
& $(2,3),\,(2,6),\,(3,2),\,(3,3),\,(4,2),\,(6,1),\, (7,1),$ & $p=2$ \\\hline
$\mathrm{PSU}_n(p^r)$ & $(3,1)$ & $p=2^s-1$ with $s\geq 2$\\
& $(3,2),\,(4,1),\,(6,1),\,(7,1)$ & $p=2$ \\\hline
$\mathrm{PSp}_n(p^r)$ & $(4,1)$ & $p=2^s-1$ with $s\geq 2$ \\
& $(4,3),\,(6,1),\,(6,2),\,(12,1)$ & $p=2$ \\\hline
\end{tabular}
 \end{center}
Each of the above groups, except $\mathrm{PSp}_6(4)$ and $\mathrm{PSp}_{12}(2)$, have disconnected prime graph by (\cite{MR1922988}, Tables 1a-1c). 

If $S=\mathrm{PSp}_6(4)$, then $|T_1|=2^6+1=65$, and $|T_2|=2^6-1=63$. Choose $x_1\in T_1$ of order $13$ and $x_2\in T_2$ of order $9$. Then $C_S(x_i)=T_i$, and $q$ divides $|x_i^S|$ for some $i\in \{1,2\}$.  Hence $S$ does not satisfy the property $P_v(p,q)$. 

If $S=\mathrm{PSp}_{12}(2)$, then by \cite{GAP4}, $|S|=2^{36}3^{8}5^37^2\cdot 11\cdot 13\cdot 17\cdot 31$. Consider the elements in $S$ of orders $31,17$ and $13$. By \cite{GAP4}, the prime divisors of their conjugacy class sizes in $S$ are as below: 
\begin{center}
\begin{tabular}{|c|c|}\hline 
Order of $a\in S$ & Prime divisors of $|a^S|$ \\\hline 
$31$ & $2,3,5,7,11,13,17$ \\
$17$ & $2,3,5,7,11,13,31$ \\ 
$13$ & $2,3,5,7,11,17,31$\\\hline 
\end{tabular}
\end{center}
Now $q\neq 2$ (since $p=2$). With $(p_1,p_2,p_3)=(31,17,13)$, we observe that among $p_1,p_2,p_3$, 

- at least two primes, say $p_i,p_j$ are coprime to $p$; 

- the conjugacy class size of $p_i$-elements or $p_j$-elements is divisible by $q$.

All the elements of $S\setminus\{1\}$ are vanishing. Hence $S$ does not satisfy the property $P_v(p,q)$. 
\end{proof}

\noindent\textbf{Theorem \ref{thma}}\label{prop P(p,q)}
Let $S$ be a non-abelian simple group and $p\neq q$ be primes dividing $|S|$. Then $S$ does not satisfy the property $P_v(p,q)$.
\begin{proof}
This follows from Lemma \ref{lemma2.3}, \ref{lemma2.4}, and \ref{lemma2.6}.
\end{proof}

\section{Proof of Theorem \ref{thmb}}
This section begins with some lemmas that will consequently be used in the proof of our main theorem.
\begin{lemma}\label{cor:vanishing_minimal_normal_combined}
Let $G$ be a finite group and $M=S_1\times \cdots \times S_k$ a minimal normal subgroup of $G$, where each $S_i\cong S$ and $S$ is a non-abelian simple group. Then the following holds.
\begin{enumerate}
\item For every prime divisor $q$ of $|S_i|$, there exists $x\in S_i$ of prime power order which is vanishing in $G$ and whose class size in $G$ is divisible by $q$.

\item For every prime divisor of $|S_i|$, there exists a $p$-element $x_i\in S_i$ which is vanishing in $G$.
\end{enumerate}
\end{lemma}

\begin{proof}
(1) Suppose $S_i$ has an irreducible character of $l$-defect zero for every prime $l$ dividing $|S_i|$. By Lemma \ref{Brough_van}, the elements of $S_i\setminus\{1\}$ are vanishing in $G$. Note that $|M|$ and $|S_i|$ have same prime divisors. Then $q\mid |x^{S_i}|$ for some $x\in S_i$ (else, by Theorem \ref{Sotomayor_van}, $S_i$ will have a normal $q$-complement).  
Since $M\unlhd G$, we have $|x^M|$ divides $|x^G|$. Therefore $q$ divides $|x^G|$, and $x$ is vanishing in $G$.

Now suppose that $S_i$ has no irreducible character of $l$-defect zero for some
prime $l$. Then $l\in\{2,3\}$, and by (\cite{MR1321575}, Corollary 2), $S_i$ is
isomorphic to one of
\begin{equation}\label{*}
A_n\ (n\ge 7),\quad M_{12},\quad M_{22},\quad M_{24},\quad J_2,\quad HS,\quad
Suz,\quad Ru,\quad Co_1,\quad Co_3,\quad \text{or } BM.
\end{equation}
By Lemma \ref{cor2} and Lemma \ref{lem_alt}, there exist $x\in S_i$ of prime
power order and  $\chi\in \mathrm{Irr}(S_i)$ such that $q$ divides $|x^{S_i}|$,
$\chi(x)=0$, and $\chi$ extends to $\mathrm{Aut}(S_i)$. 
By Lemma \ref{lem2}, $\chi\times\cdots\times \chi\in \mathrm{Irr}(M)$
extends to $G$, and vanishes at $x$. Moreover, $q$ divides $|x^{S_i}|=|x^M|$, which divides $|x^G|$. This proves $(1)$.
\vskip2mm
$(2)$ Suppose $S_i$ has an irreducible character of $p$-defect zero, then so does $M$. Hence, by
Lemma \ref{Brough_van}, every $p$-element of $S_i$ is
vanishing in $G$. 

Suppose $S_i$ has no irreducible character of $p$-defect zero. Then
$p\in\{2,3\}$, and by (\cite{MR1321575}, Corollary 2), the group $S_i$ is
isomorphic to one of the groups in \ref{*}.

If $p=2$, then Lemma \ref{lem_alt} and
Lemma \ref{cor2} (see tables in the proof) implies that there exists a $2$-element $x_i\in S_i$ and
a character $\chi\in \mathrm{Irr}(S_i)$ such that $\chi(x_i)=0$ and $\chi$ extends to $\mathrm{Aut}(S_i)$. By Lemma \ref{lem2}, $\chi\times\cdots\times \chi\in \mathrm{Irr}(M)$ extends to $G$. Therefore, $x_i$ is vanishing in $G$.

Let $p=3$. If $S_i\cong A_n(n\geq 7),\ M_{12},\, J_2,\, Suz,\, Co_1$ or  $Co_3$, by Lemma \ref{lem_alt} and Lemma \ref{cor2} (see tables in the proof), there is a $3$-element in $S_i$ and a character $\chi\in \mathrm{Irr}(S_i)$ such that $\chi(x_i)=0$ and $\chi$ extends to $\mathrm{Aut(S_i)}$. If, $S_i\cong HS,\ Ru,\ BM,\ M_{22},\ \text{or } M_{24}$, then following the notations in \cite{MR827219}, choose $x_i$ from the conjugacy class
$3A$ of these groups and the following characters
$$
\chi_9,\quad \chi_5,\quad \chi_{78},\quad \chi_6,\quad \chi_{20},
$$
respectively. Each of these characters extends to $\mathrm{Aut}(S_i)$ and vanishes on
the corresponding element $x_i$. By Lemma \ref{lem2}, $\chi\times\cdots\times \chi\in \mathrm{Irr}(M)$ extends to $G$, and $x_i$ remains vanishing in $G$.
\end{proof}

\begin{lemma}\label{lemab}
If $G$ satisfies the property $P_v(p,q)$ with primes $p\neq q$, then any minimal normal subgroup $M$ of $G$, whose order is divisible by $q$, is abelian. 
\end{lemma}
\begin{proof}
If possible, let $M = S_1 \times \cdots \times S_k$ be a \textit{non-abelian} minimal normal subgroup of $G$, where each $S_i$ is isomorphic to a simple group $S$, and $q \mid |M|$.

\noindent
\textbf{Case 1:} Suppose that $S$ has no irreducible character of $l$-defect zero for some prime $l$.

\noindent
By Lemma \ref{cor2} and Lemma \ref{lem_alt}, there is a $p'$-element $g \in S$ of prime power order and $\chi \in \mathrm{Irr}(S)$ such that $q \mid |g^S|$ (which divides $|g^G|$) and  $\chi(g) = 0$. By Lemma \ref{lem2}, $\chi\times\cdots\times \chi\in \mathrm{Irr}(M)$ extends to $G$, hence $g$ is vanishing in $G$. Thus, $G$ does not satisfy the property $P_v(p,q)$, a contradiction.

\noindent
\textbf{Case 2:} Suppose that $S$ has an irreducible character of $l$-defect zero for all primes $l$.

Then this property is also satisfied by $M$. By (\cite{MR1280461}, Theorem 8.17),  all the elements in $M \setminus \{1\}$ are vanishing in $M$. By Lemma \ref{Brough_van}, all elements in $M \setminus \{1\}$ are vanishing in $G$.

If $p$ divides $|M|$, then for any $p'$-element $x\in S_i$ of prime power order, since $x$ is vanishing in $G$, it follows that $q \nmid |x^S_i|$. Hence $S_i$ satisfies the property $P_v(p,q)$, a contradiction by Theorem \ref{prop P(p,q)}.

If $p$ does not divide $|M|$, then all the elements of $M$ of prime power order are $p'$-elements, they are vanishing in $G$, and their conjugacy class sizes in $G$ (and hence in $M$) are coprime to $q$. This contradicts Lemma \ref{cor:vanishing_minimal_normal_combined} (1).

The contradictions in all the cases show that $S_i$ and hence $M$ must be abelian ($q$-group).\end{proof}

The following Lemma will be crucial in the proof of our main theorem.

\begin{lemma}\label{lemn}
Let $G$ be a finite group, $N\trianglelefteq G$ and $M=S_1\times \cdots \times S_k$ ($k\ge 2$) a minimal non-abelian normal subgroup of $G$ contained in $N$, where $S_i\cong S$ is a simple group, not isomorphic to $\mathrm{PSL}_2(l)$. Suppose that $N/M$ is solvable, and non-trivially permutes $\{S_1,\ldots, S_k\}$. Then for every odd prime $p$ dividing $|G|$, there exists a $p'$-element of prime power order $x\in M$, which is vanishing in $G$, such that $|x^G|$ is divisible by every prime divisor of $|N/M|$.    
\end{lemma}

\begin{proof} First, we prove the following 

\noindent\textbf{Claim:} There exist $a,b\in S$, whose orders are powers of the same prime, that belong to different $\mathrm{Aut}(S)$-orbit and are vanishing in $G$.

\vskip5mm\noindent
\textbf{Case 1.} $S$ has an element of order $4$.

 By Lemma \ref{cor:vanishing_minimal_normal_combined},  for the prime divisor $2$ of $|S|$, there is a $2$-element $a$ in $S$ which is vanishing in $G$. Note that $a$ is a $p'$-element (since $p>2$). Choose $b\in S$ a $2$-element whose order is different than $o(a)$; this is possible since $S$ contains elements of order $4$.

\vskip5mm\noindent
\textbf{Case 2.} $S$ has no element of order $4$:

Then by (\cite{MR569209}, Page. 485 ), $S$ is one of the following: $S \cong \mathrm{PSL}_2(l) $ for some $l$, or  $S\cong J_1$, or $S\cong { }^2G_2(3^{2n+1})$, $n\ge 1$ (Ree  group). By assumption, $S\not\cong \mathrm{PSL}_2(l)$. 
If $S\cong J_1$, then by \cite{MR827219} for $p\neq 5$, take $a,b\in S$, both of order $5$ but from different conjugacy classes, and if $p=5$, take $a,b\in S$ of order $19$ from different conjugacy classes in $S$. In each of these cases, $a,b$ are $p'$-elements in different $\mathrm{Aut}(S)$-orbits (since $\mathrm{Aut}(J_1)=\mathrm{Inn}(J_1)$); their orders are coprime to $6$, they are vanishing in $G$. 

Let $S\cong {^2G_2}(3^{2n+1})$, $n\geq 1$. Note that $\mathrm{Out}(S)\cong C_{2n+1}$ (generated by field automorphisms of order $2n+1$). We now consider the following subcases:

\textbf{Case 2(A)}: If $p\neq 3$, then we choose $a$ of order $3$ and $b$ of order $9$. Since ${}^2G_2(3^{2n+1})$ has irreducible character of defect $0$ for every prime, $a,b$ are vanishing in $G$.

\vskip5mm
\textbf{Case 2(B): $p=3$.}

By (\cite{MR2507573}, Theorem 3.1), $S$ has two cyclic maximal tori
$T_{1}$ and $T_{2}$ of respective orders
$$
|T_{1}|=q-\delta+1
\qquad\text{and}\qquad
|T_{2}|=q+\delta+1.
$$ where $\delta=\sqrt{3q}=3^{n+1}$ and $q=3^{2n+1}$.
Notice that
$$
(q-\delta+1)(q+\delta+1)
=
q^2-q+1
=
\Phi_6(q).
$$
 For each
$\varepsilon\in\{1,2\}$, by Theorem \ref{zsigmondy}, there is a prime
$r_{\varepsilon}\mid q+\varepsilon\delta+1$
which is a primitive prime divisor of
$3^{6(2n+1)}-1$ and consequently
$6(2n+1)\mid r_{\varepsilon}-1.$ Since $\gcd(q-\delta+1,q+\delta+1)=1$, $r_{1}\neq r_{2}$. Now, $r_{\varepsilon}-1 \ge 6(2n+1)$, we get $r_{\varepsilon}-1>6(2n+1)$ for some $\varepsilon\in\{1,2\}$; fix such a sign $\varepsilon$, and write
$$
T=T_{\varepsilon}
\qquad\text{and}\qquad
r=r_{\varepsilon}.
$$
Consider $x\in T$ of order $r$. By $\S 3$ in \cite{MR3232797} , we can assume that $S$ is a subgroup of $\mathrm{SL}(V)$, where $\dim_{\mathbb{F}_q}(V)=7$ and $S$ preserves a non-degenerate symmetric bilinear form on $V$. Also, by (\cite{MR3232797}, Lemma 3.4), $1$ is an eigenvalue of every $g\in S$. Since $|\mathrm{Out}(S)|=2n+1$, the $\mathrm{Aut}(S)$-orbit of $x$ has at most $6(2n+1)$ non-trivial eigenvalues in $\overline{\mathbb{F}}_q^{\times}$, and they have order $r$. Since $r-1=\varphi(o(x))>6(2n+1)$, there is $1<i<r$ such that  an eigenvalue of $x^i$ is not the eigenvalue of any member in $\mathrm{Aut}(S)$-orbit of $x$. It follows that $x$ and $x^i$ are $p'$-elements that belong to distinct $\mathrm{Aut}(S)$-orbit and have order $r$. Let $a=x$ and $b=x^i$. Since $o(a),o(b)$ are coprime to $6$, by (\cite{MR3500775}, Lemma 2.2) $a$ and $b$ are vanishing in $G$. This completes the proof of the claim.

\vskip5mm
Consider the action of $N/M$ on $\Omega:=\{1,\ldots, k\}$ by identifying $i$ with $S_i$. By (\cite{MR1785438}, Corollary 4), there exist  disjoint subsets $\Gamma$ and $\Delta$ of $\Omega$ such that every prime divisor of $|N/M|$ also divides 
$$
|N/M : \mathrm{Stab}_{N/M}(\Gamma)\cap \mathrm{Stab}_{N/M}(\Delta)|.
$$
Now fix an isomorphism $\phi_i: S\rightarrow S_i$. With $a$ and $b$ constructed in the above claim, let $a_i=\phi_i(a)$, $b_i=\phi_i(b)$. Consider $x=(x_1,\ldots, x_k)\in S_1\times \cdots \times S_k$ where  $x_i=1$ if $i\notin \Gamma\cup\Delta$ and 
$$
x_i=
\begin{cases} 
a_i & \mbox{ if } i\in \Gamma \\
b_i & \mbox { if } i\in \Delta.\end{cases}
$$ 
Recall that by (\cite{MR1357169}, Theorem 3.3.20) for a non-abelian simple group $S$, any automorphism of $S\times \cdots \times S$ is composition of 

- an automorphism which stabilizes each component $S$ (set-wise); 

- an automorphism which permutes the components.

The construction of $x$, using $a$ and $b$, shows that if $g\in N$ centralizes $x$, then $g$ must stabilize $\Gamma$ and $\Delta$, i.e. 
$$
C_{N}(x)M/M \subseteq \mathrm{Stab}_{N/M}(\Gamma)\cap \mathrm{Stab}_{N/M}(\Delta).
$$ 
Hence $[N/M:C_{N}(x)M/M]$ is divisible by $[N/M:\mathrm{Stab}_{N/M}(\Gamma)\cap \mathrm{Stab}_{N/M}(\Delta)]$, which is divisible by every prime divisor of $|N/M|$. Hence, $|x^G|$ is divisible by every prime divisor of $|N/M|$. \end{proof}

\vskip5mm
Some finite simple groups possess $S$ a coprime automorphism (i.e. automorphism of coprime order). Dolfi et al. in \cite{MR2469367} proved that the fixed point subgroup of such an automorphism has order not divisible by at least one prime divisor of $|S|$. We obtain a generalization of this result, which will be used in our second main theorem.

\begin{lemma}\label{lem2.9}
Let $S$ be a non-abelian simple group that has an automorphism $\alpha$ of coprime order. Then there exist two distinct odd prime divisors $t_i$ ($i=1,2$) of $|S|$ such that $(t_i,|C_S(\alpha)|)=1$.
\end{lemma}
\begin{proof}
It is well-known (\cite{MR3753581}, Corollary 2.6.2) that the only non-abelian simple groups that possess a coprime automorphism are (some) \textit{simple groups of Lie type}, say $S=G(r^a)$, for a prime $r$ and $a\geq 1$; moreover, the coprime automorphisms are (conjugate to) field automorphisms of $S$, and have odd order (since $|S|$ is even). 

Let $\alpha\in\mathrm{Aut}(S)$ be of order $p$ (prime), which is induced by a member in $\mathrm{Gal}(\mathbb{F}_{r^a}/\mathbb{F}_r)$. We conclude that $S=G(r^{pf})$ and $C_S(\alpha)=G(r^f)$ where $a=pf$. For simplicity, let $r^f=q$. Note that $r\neq p$ (since $r||S|$). Thus, 
$$
S=G(q^p), \hskip5mm C_S(\alpha)=G(q),
$$
where $\alpha$ is coprime automorphism of $S$ of order $p$. If $S$ is not a Suzuki group, it is known that $|S|$ is divisible by $2\cdot 3$, so we must have $p>3$.

In the table given below, we list integers $m,n$ of the form $q^{ip}-1$ ($i\ge 1$) which divide $|S|$. Then we choose $t_1,t_2$ as primitive prime divisors of $m,n$ respectively; by Theorem \ref{zsigmondy}, they exist since $ip\notin \{2,6\}$ for all the cases of $S$, except for $S={}^2B_2(q^p)$, a Suzuki group. If $S={}^2B_2(q^p)$, then primitive prime divisors of $q^{4p}-1$ and of $q^p-1$ exists (since $p$ is odd and $4p\notin \{2,6\}$).

\vskip2mm
\begin{center}
\begin{tabular}{|c|c|c|}
\hline
\textbf{Simple group $S$} ($p$ odd) & $m$ & $n$ \\
\hline
$\mathrm{F}_4(q^p)$, ${}^2\mathrm{E}_6(q^{2p})$, $\mathrm{E}_6(q^p)$, $E_7(q^p)$ & $q^{12p}-1$ & $q^{8p}-1$ \\\hline

$\mathrm{E}_8(q^p)$ & $q^{30p}-1$ & $q^{24p}-1$ \\
\hline
${}^2\mathrm{F}_4(q^p)$ ($q=2^{2n+1},n\geq 1$) & $q^{6p}+1$ & $q^{3p}+1$ \\
\hline

$\mathrm{G}_2(q^p)$ ($q\geq 3$), $^3\mathrm{D}_4(q^{3p})$ & $q^{3p}-1$ & $q^{2p}-1$ \\
\hline
$\mathrm{PSL}_n(q^p)$ ($n\geq 2$) & $q^{np}-1$ & $q^{(n-1)p}-1$ \\
\hline
$\mathrm{PSp}_{2n}(q^p)$ ($n\geq2$), &  &  \\

$\mathrm{P}\Omega_{2n+1}(q^p)$ $(n\geq 3$, $q$-odd), &  $q^{2n p}-1$ & $q^{(2n-2)p}-1$  \\
$\mathrm{P}\Omega_{2n}^{-}(q^p)$ ($n\geq 4$) &  &  \\
\hline
$\mathrm{P}\Omega_{2 n}^{+}(q^p)$ ($n\geq 4$) & $q^{n p}-1$ & $q^{(2n-2)p}-1$ \\

$\mathrm{PSU}_n(q^{2p})$ ($n\geq 3$) &  & \\
\hline

${}^2\mathrm{B}_2(q^p)$ ($q=2^{2n+1},n\geq 1$) & $q^{4p}-1$ & $q^{p}-1$ \\
\hline
${}^2\mathrm{G}_2(q^p)$ ($q=3^{2n+1},n\geq 1$) & $q^{6p}-1$ & $q^{p}-1$ \\
\hline
\end{tabular}
\end{center}
\end{proof}

We now prove the second main theorem of our paper, which generalises a result of Dolfi and Lucido in \cite{MR1826493}. 

\vskip3mm
\noindent\textbf{Theorem \ref{thmb}} If $G$ is a finite group, satisfying the property $P_v(p,q)$ with primes $p\neq q$, and $p>2$, then $O^{q'}(G)$ is solvable.
 
\begin{proof} We assume that $q$ divides $|G|$ (otherwise, $O^{q'}(G)=1$).  By Proposition \ref{prop P(p,q)}, $G$ is not a simple group. We proceed by induction on $|G|$. Hence, we assume that the theorem is true for all the groups of order $<|G|$. 

\noindent\textbf{Claim 1.} For any minimal normal subgroup $M$ of $G$, $O^{q'}(G/M)$ is solvable. 

Let $\overline{G}:=G/M$. If $q\nmid |\overline{G}|$, then $O^{q'}(\overline{G})=1$, proving the claim. Hence, we assume that $q$ divides $|\overline{G}|$.

If $\overline{G}$ has no vanishing $p'$-elements of prime power order, then in particular, all the $q$-elements in $\overline{G}$ are non-vanishing. By Theorem \ref{Thm_normal}, $\overline{G}$ has normal Sylow $q$-subgroup, it must be then $O^{q'}(\overline{G})$, hence it is solvable. 

If $\overline{G}$ has vanishing $p'$-elements of prime power order, then by induction $O^{q'}(\overline{G})$ is solvable.

\noindent\textbf{Claim 2.} Any minimal normal subgroup of $G$ contained in $O^{q'}(G)$ is abelian.

If possible, let $M=S_1\times \cdots \times S_k$ $(k\ge 1$) be a minimal normal subgroup of $G$, such that $M\subseteq O^{q'}(G)$ and $S_i\cong S$ is non-abelian simple group. First we show that $O^{q'}(G)/M$ acts trivially on $\{S_1,\ldots, S_k\}$ (i.e each $S_i\trianglelefteq O^{q'}(G)$). This is clearly true if $k=1$; we assume that $k\ge 2$.  

\textbf{Case A.} $S_i\not\cong \mathrm{PSL}_2(l)$ for $l\ge 4$. 

Since $M$ is non-abelian, by Lemma \ref{lemab}, $q\nmid |M|$. Note that $O^{q'}(G)/M\subseteq O^{q'}(G/M)$, which is solvable (by Claim 1). If the action of $O^{q'}(G)/M$ on $\{S_1,\ldots, S_k\}$ is non-trivial, then by Lemma \ref{lemn}, there exists a $p'$-element $x\in M$ of prime power order which is vanishing in $G$, such that $|x^G|$ is divisible by every prime divisor of $|O^{q'}(G)/M|$, hence by $q$, a contradiction to the hypothesis on $G$.

\vskip2mm
\textbf{Case B.} $S_i\cong \mathrm{PSL}_2(l)$ for some $l\ge 4$. 

If the action of $O^{q'}(G)/M$ on $\Omega:=\{S_1,\ldots, S_k\}$ is non-trivial, then by (\cite{MR898721}, Lemma 7), there exists $\Delta\subseteq \Omega$ such that $|\mbox{Stab}_{O^{q'}(G)/M}(\Delta)|$ is a $\{2,3\}$-group. By Lemma \ref{cor:vanishing_minimal_normal_combined} (2), there exists a $p'$-element $a_i\in S_i$ of prime power order which is vanishing in $G$. Consider $x=(x_1,\ldots, x_k)\in S_1\times \cdots \times S_k$ where 
$$
x_i=
\begin{cases} 
1 & \mbox{ if } S_i\notin \Delta \\
a_i & \mbox { if } S_i\in \Delta.\end{cases}
$$  
Clearly, $x$ is a vanishing $p'$-element of prime power order in $G$. If $g\in C_N(x)$, then by construction of $x$, $g$ must stabilize $\Delta$, i.e. $gM\in  \mathrm{Stab}_{O^{q'}(G)/M}(\Delta)$. Now, $|S_i|=|\mathrm{PSL}_2(l)|$  ($l\ge 4$) is always divisible by $2\cdot 3$. Since $q \nmid |S_i|$ and $\mathrm{Stab}_{O^{q'}(G)/M}(\Delta)$ is \{2,3\} group, it follows that $q$ does not divide the order of $g$ in $G$. 

Thus $q\nmid |C_{O^{q'}(G)}(x)|$, we must have $q\mid |x^{O^{q'}(G)}|$, which divides $|x^G|$, a contradiction to the hypothesis on $G$.

\vskip2mm
The contradiction in both cases shows that $O^{q'}(G)$ stabilizes each $S_i$.  

\vskip2mm
Fix $i\in \{1,\ldots,k\}$ and fix a $q$-Sylow subgroup $Q$ of $G$.   For every prime divisor $l$ of $|S_i|$, other than $p$, there is an $l$-element $y_l\in S_i$ which is vanishing in $G$. Hence $C_G(y_l)\supseteq Q'$ for some Sylow $q$-subgroup $Q'$ of $G$. Since all the Sylow $q$-subgroups of $G$ are in $O^{q'}(G)$, we can write $Q'=hQh^{-1}$ for some $h\in O^{q'}(G)$. Hence, $C_G(h^{-1}y_lh)\supseteq Q$, where $h^{-1}y_lh\in S_i$ (since $S_i\trianglelefteq O^{q'}(G)$).  This shows that $Q$ centralizes an $l$-element in $S_i$ for every prime divisor of $|S_i|$ except $p$. 

If the action of $Q$ on $S_i$ is non-trivial, then some $\sigma\in Q$ has a non-trivial and coprime action on $S_i$. From the previous paragraph, $|C_{S_i}(\sigma)|$ is divisible by all the primes dividing  $|S_i|$, except possibly one (namely $p$); this is a contradiction by Lemma \ref{lem2.9}. Thus $Q$ acts trivially on $S_i$ for every $i$, i.e. $C_G(M)\supseteq Q$. Since $M\trianglelefteq G$, $C_G(M)\trianglelefteq G$, hence $C_G(M)\supseteq O^{q'}(G)\supseteq M$, a contradiction (since $M$ is non-abelian).  This proves Claim 2. 

\vskip3mm 
We now prove that $O^{q'}(G)$ is solvable. Let $M$ be a minimal normal subgroup of $G$ with $M\subseteq O^{q'}(G)$. Then by above claims, $M$ is abelian, and $O^{q'}(G/M)$ is solvable. Since $O^{q'}(G/M)=O^{q'}(G)/M$, it follows that $O^{q'}(G)$ is solvable. 
\end{proof}

 \bibliographystyle{plain}  
 \bibliography{biblio}   
\vskip2mm\noindent 
\textsuperscript{\textdagger}\textsc{Sonakshee Arora}, Department of Mathematics, Indian Institute of Technology Jammu, Jagti, NH-$44$, PO Nagrota, Jammu-$181221$, J$\&$K,  India. 

\noindent Email: \texttt{\color{blue}sonakshee.arora@iitjammu.ac.in}
\vskip2mm\noindent 
\textsuperscript{$\ddagger$}\textsc{Rahul Dattatraya Kitture}, Department of Mathematics, Indian Institute of Technology Jammu, Jagti, NH-$44$, PO Nagrota, Jammu-$181221$, J$\&$K, India. 

\noindent Email: \texttt{\color{blue}rahul.kitture@iitjammu.ac.in}

\end{document}